\documentclass[12pt]{amsart}
\usepackage[T1]{fontenc}
\usepackage[a4paper,margin=1in]{geometry}
\usepackage{graphicx}
\usepackage{amssymb}
\usepackage{stmaryrd}
\usepackage{enumitem}
\usepackage{tikz}
\usetikzlibrary{arrows.meta,calc}
\usepackage{subcaption}

\usepackage{float}
\usepackage{url}
\usepackage{mathtools}
\usepackage{amsmath,amsfonts,amsthm}
\usepackage{mathrsfs} 
\newcommand{\case}[1]{\paragraph*{Case #1:}}

\newtheorem{theorem}{Theorem}[section]
\newtheorem{lemma}[theorem]{Lemma}
\newtheorem{proposition}[theorem]{Proposition}
\newtheorem{corollary}[theorem]{Corollary}
\newtheorem{definition}{Definition}
\newtheorem{example}{Example}[section]

\usepackage{multirow}

\newcommand{\QFR}{\mathrm{QFR}}
\newcommand{\PST}{\mathrm{PST}}

\usepackage[colorlinks,
linkcolor=blue,
anchorcolor=blue,
citecolor=blue
]{hyperref}

\author{D. Abdullah}
\address{\textbf{Duaa Abdullah} Department of Discrete Mathematics, Moscow Institute of Physics and Technology}
\email{abdulla.d@phystech.edu}
\thanks{ }

\title[Quantum Fractional Revival]{Quantum Fractional Revival and Entanglement Entropy in Unitary Cayley Graphs}

\date{}
\begin{document}

\begin{abstract}
This paper extends the theory of quantum fractional revival (QFR) on unitary Cayley graphs
$X=(V(\mathbb{Z}_n),E(S))$ in several directions that remained unresolved in previous work.
First, we investigate QFR with respect to the Laplacian matrix Hamiltonian in addition to the
adjacency matrix Hamiltonian. In particular, we prove that for regular graphs the two models
differ only by a global phase factor, and we determine the conditions under which the Laplacian
framework independently admits QFR.
Second, for unitary Cayley graphs of order $n=2p$, where $p$ is an odd prime, we derive an
explicit closed-form expression for the minimum revival time,
$t^{*}=\frac{2\pi}{p},$
and show that the associated revival amplitudes are given by
\[
\alpha=\cos\!\left(\frac{2\pi}{p}\right), \qquad
\beta=-i\sin\!\left(\frac{2\pi}{p}\right).
\]
Third, we provide a complete characterization of strongly cospectral vertex pairs in
$X=(V(\mathbb{Z}_n),E(S))$ through the arithmetic structure of $\mathbb{Z}_n$, establishing that
strong cospectrality is equivalent to antipodality whenever $n$ is twice a prime.
Finally, we compute the von Neumann entanglement entropy generated by QFR for all admissible
graphs, thereby obtaining a collection of quantum information measures and proving that the
entropy depends solely on the revival amplitudes $|\alpha|$ and $|\beta|$.

\end{abstract}
\maketitle

\noindent\rule{15.9cm}{1.0pt}


\medskip


\noindent\rule{15.9cm}{1.0pt}

\section{Introduction}
Quantum fractional revival $\QFR$ in graph networks is a central topic in the theory of
continuous-time quantum walks (CTQWs)~\cite{genest2016,bernard2018}, with direct applications to quantum communication,
quantum entanglement generation, and the design of quantum spin chains. The foundational
work of Soni, Choudhary, and Singh~\cite{soni2025original} established that $\QFR$ exists in
 $X = (V(\mathbb{Z}_n), E(S))$ where it holds if and only if $n=2k$ and $k\in \mathbb{N}$. Using the adjacency matrix as the Hamiltonian. That paper characterized QFR in terms of
the spectral theory of the adjacency matrix and the Ramanujan sum function~\cite{chan2019, chan2022fundamentals, chan2021laplacian}.

Despite these achievements, several fundamental questions remained open. The paper relied
exclusively on the adjacency matrix as the Hamiltonian, leaving the Laplacian setting
entirely unexplored. No closed-form formula was given for the revival time $t^*$ as a function
of $n$. The characterization of strongly cospectral vertex pairs was incomplete beyond the
antipodal case. Therefore, no quantum information metrics such as entanglement entropy were
computed, leaving the communication utility of the graphs unquantified (furthermore, see~\cite{nielsen2000, saxena2007, so2006}).

\subsection{Preliminaries}
For any simple graph $G = (V, E)$ with $n$ vertices and $m$ edges, the concept of \emph{adjacency matrix} refer to matrix $A$ for $\{u,v\} \in E$ where $A_{uv} = 1$  and $0$ otherwise. Define \emph{Laplacian matrix} as
$L = D - A$, where the degree matrix is $D = \mathrm{diag}(d(v_1), \ldots, d(v_n))$.   
For a $k$-regular graph, $D = kI$ and hence $L = kI - A$.

The \emph{evolution operator} of a continuous-time quantum walk driven by a Hamiltonian $H$
is:
\begin{equation}\label{eq:evolution}
  U_H(t) =\mu^{-iHt}.
\end{equation}

\begin{definition}\label{def:qfr}
  A graph $G$ admits \emph{quantum fractional revival} (QFR) between vertices $u$ and $v$ at
  time $t > 0$ with Hamiltonian $H$ if
  \begin{equation}\label{eq:qfr_def}
    U_H(t)\, \mu_u = \alpha\, \mu_u + \beta\, \mu_v,
    \qquad |\alpha|^2 + |\beta|^2 = 1,\quad \beta \neq 0,
  \end{equation}
  where $\mu_u, \mu_v$ are standard basis vectors and $\alpha, \beta \in \mathbb{C}$. 
\end{definition}
Hence, we consider
  QFR reduces to \emph{perfect state transfer} (PST) when $\alpha = 0$, and is called
  \emph{balanced} when $|\alpha| = |\beta| = 1/\sqrt{2}$. Determine $X = (V(\mathbb{Z}_n), E(S))$ has vertex set $V(\mathbb{Z}_n)$ and
edge set $E(S)= \{u \in V(\mathbb{Z}_n) \mid  \gcd(u, n) = 1\}$, the group of units $V(\mathbb{Z}_n)$. In this case, for any vertices $\{u,v\}\subseteq V(\mathbb{Z}_n)$, we say the vertex $u$ adjacent to vertex $v$ if and only if $\gcd(d(u)-d(v), n)=1$.

Furthermore, see spectral properties of $X$ are as follows (see~\cite{chan2021pretty, chan2020, klotz2007,liu2022}).
The eigenvalues of the adjacency matrix are the Ramanujan sums
\begin{equation}\label{eq:eigenvalues}
  \lambda_\kappa=\gamma\!\left(\frac{n}{\gcd(\kappa,n)}\right)\dfrac{\varphi(n)}{\varphi\!\left(\dfrac{n}{\gcd(\kappa,n)}\right)},
\end{equation}
where $\kappa= 0, 1, \ldots, n-1,$ and  $\gamma$ is the M\"{o}bius function and $\varphi$ is Euler's totient function.
The corresponding eigenvectors are $$v_\kappa=n^{-1/2}[1, \omega_n^{\kappa}, \omega_n^{2\kappa}, \ldots, \omega_n^{(n-1)\kappa}]^T,$$
where $\omega_n =\mu^{2\pi i/n}$.
Thus, we consider a graph is $\varphi(n)$-regular, connected, and its adjacency matrix is a circulant
with integer (integral) eigenvalues.
The evolution operator under the adjacency Hamiltonian is
\begin{equation}\label{eqq01evolutionadj}
  U_A(t) = \frac{1}{n}\sum_{d=0}^{n-1} \mu^{i\lambda_\kappa t} E_\kappa,
\end{equation}
where $E_\kappa= n\, v_\kappa v_\kappa^*,$
with matrix terms
\begin{equation}\label{eqq02evolutionadj}
  \alpha_A = U_A(t)_{u,u} = \frac{1}{n}\sum_{\kappa=0}^{n-1} \mu^{i\lambda_\kappa t}, \qquad
  \beta_A  = U_A(t)_{v,u} = \frac{1}{n}\sum_{\kappa=0}^{n-1} \mu^{i\lambda_\kappa t}\,\omega_n^{(v-u)\kappa}.
\end{equation}
Figure~\ref{fig01cayley} illustration the unitary Cayley graphs (see~\cite{klotz2007, liu2022, wang2024semi, soni2025mixed}) for several vertices $n = 2, 3, 4, 5, 6$. Note that for prime $n$ the graph is complete, while for $n=4$ it is $K_{2,2}$
           (complete bipartite) and for $n=6$ it is a circulant of degree~2.

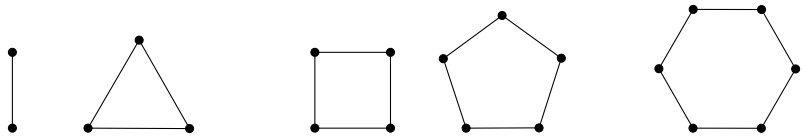
\begin{figure}[H]
\centering
\begin{tikzpicture}
\draw   (12,10)-- (12,9);
\draw   (13.676555582323173,10.160564151981237)-- (13,9);
\draw   (13,9)-- (14.343355829498881,8.994367754626577);
\draw   (14.343355829498881,8.994367754626577)-- (13.676555582323173,10.160564151981237);
\draw   (16,10)-- (16,9);
\draw   (16,9)-- (17,9);
\draw   (17,9)-- (17,10);
\draw   (17,10)-- (16,10);
\draw   (17.6984177253741,9.916197844555347)-- (18,9);
\draw   (18,9)-- (18.964549978341594,9.00370108330478);
\draw   (18.964549978341594,9.00370108330478)-- (19.25909237417877,9.922186323137675);
\draw   (19.25909237417877,9.922186323137675)-- (18.47657960759238,10.486140336214724);
\draw   (18.47657960759238,10.486140336214724)-- (17.6984177253741,9.916197844555347);
\draw   (20.549357978676127,9.783832761366323)-- (21,9);
\draw   (21,9)-- (21.90414009432368,8.998351057796079);
\draw   (21.90414009432368,8.998351057796079)-- (22.357638167323486,9.780534876958484);
\draw   (22.357638167323486,9.780534876958484)-- (21.906996145999614,10.564367638324809);
\draw   (21.906996145999614,10.564367638324809)-- (21.002856051675934,10.566016580528728);
\draw   (21.002856051675934,10.566016580528728)-- (20.549357978676127,9.783832761366323);
\begin{scriptsize}
\draw [fill=black] (12,10) circle (1.5pt);
\draw [fill=black] (12,9) circle (1.5pt);
\draw [fill=black] (13.676555582323173,10.160564151981237) circle (1.5pt);
\draw [fill=black] (13,9) circle (1.5pt);
\draw [fill=black] (14.343355829498881,8.994367754626577) circle (1.5pt);
\draw [fill=black] (16,10) circle (1.5pt);
\draw [fill=black] (16,9) circle (1.5pt);
\draw [fill=black] (17,9) circle (1.5pt);
\draw [fill=black] (17,10) circle (1.5pt);
\draw [fill=black] (17.6984177253741,9.916197844555347) circle (1.5pt);
\draw [fill=black] (18,9) circle (1.5pt);
\draw [fill=black] (18.964549978341594,9.00370108330478) circle (1.5pt);
\draw [fill=black] (19.25909237417877,9.922186323137675) circle (1.5pt);
\draw [fill=black] (18.47657960759238,10.486140336214724) circle (1.5pt);
\draw [fill=black] (20.549357978676127,9.783832761366323) circle (1.5pt);
\draw [fill=black] (21,9) circle (1.5pt);
\draw [fill=black] (21.90414009432368,8.998351057796079) circle (1.5pt);
\draw [fill=black] (22.357638167323486,9.780534876958484) circle (1.5pt);
\draw [fill=black] (21.906996145999614,10.564367638324809) circle (1.5pt);
\draw [fill=black] (21.002856051675934,10.566016580528728) circle (1.5pt);
\end{scriptsize}
\end{tikzpicture}
\caption{$X=(V(\mathbb{Z}_n),E(S))$ for $n=2,3,4,5,6$ vertices respectively.}
  \label{fig01cayley}
\end{figure}
\subsection{Problem Statement}
This paper addresses all four of these gaps.
\begin{enumerate}
    \item Prove that for regular graphs, Laplacian QFR
and adjacency $\QFR$ are related by a phase factor arising from the degree $\varphi(n)$.
\item Determine through this paper the sharp 
revival time satisfying $t^* = 2\pi/p$ for the number of vertices $n = 2p$ by considering the case  $\alpha = \cos(2\pi/p)$ and
$\beta = -i\sin(2\pi/p)$.
\item Completely characterize
strongly cospectral pairs for $n = 2p$. 
\item Compute
the von Neumann entropy and identify when perfect state transfer (PST) occurs as a degenerate
case of $\QFR$.
\end{enumerate}

\section{On Laplacian Quantum Fractional Revival}~\label{secreval}
In this section, we establish a graph $X = (V(\mathbb{Z}_n),E(S))$ is $\varphi(n)$-regular, where  $L = \varphi(n)I - A$. This gives
\begin{equation}~\label{eqq01Laplacian}
  U_L(t)=\mu^{-i\varphi(n)t}\, \overline{U_A(t)},
\end{equation}
where $\overline{U_A(t)}$ denotes the complex conjugate which equals $U_A(-t)$ by unitarity.

\subsection{PST as a Special Case and Periodicity} 
In this subsection, we provide two complementary conclusions that contextualise $\QFR$ within the hierarchy of quantum state transfer events. First, we determine when PST (the limiting situation $\alpha = 0$, $|\beta| = 1$) happens in unitary Cayley graphs. Second, we find the least period of the quantum walk, which determines the timeframe for the evolution to recur.

\begin{proposition}\label{prop001pst}
Let $X$ be an unitary Cayley graph. Then, perfect state transfer between the antipodal pair $(0,\, n/2)$ in the 
  $X$, under the adjacency Hamiltonian, occurs if and only if
  $n \in \{2, 4\}$. 
\end{proposition}
\begin{proof}
For determining the value of  PST we start from $u=0$ to $v=n/2$ at time $t$ where $\alpha = U_A(t)_{0,0} = 0$ and
  $|\beta| = |U_A(t)_{n/2,\, 0}| = 1$. Then,
  \[
    \alpha = \frac{1}{n}\sum_{\kappa=0}^{n-1} e^{i\lambda_{\kappa} t}.
  \]
\case{1} If $n=2$. The eigenvalues are $\lambda_0 = 1$ and $\lambda_1 = -1$.
  Then $\alpha=\cos t$. Assume $\alpha = 0$ it follows that
  $t = \pi/2$, at which point
  $\beta=-i$, where $|\beta| = 1$ and PST holds.
  
\case{2} If $n=4$.  The eigenvalues are $\lambda_0 = 2$, $\lambda_1 = \lambda_3 = 0$,
  and $\lambda_2 = -2$. Then
  \[
    \alpha=\frac{1+ \cos 2t}{2}, \qquad \beta=\frac{i\sin 2t}{2}.
  \]
Hence, at $t=\pi/2$, $\beta = i\sin\pi/2 = 0$. In this case, the eigenvalue multiset is $\{2, 0, 0, -2\}$ for $\kappa= 0,1,2,3$ respectively, then
  $\lambda_1 = 0$, $\lambda_2 = -2$, $\lambda_3 = 0$. Thus, 
  \[
  \beta=\frac{\cos 2t - 1}{2},
  \]
and at $t = \pi/2$, we have $\beta=-1$, and PST holds at $t = \pi/2$.

\case{3} No PST for $n = 2p$, $p \geq 3$. According to Theorem~\ref{Thm2revivaltime}, 
  at the revival time $t^* = 2\pi/p$ we have $\alpha = \cos(2\pi/p)$ and by considering $2\pi/p < \pi/2$
  for all $p \geq 5$. Since $2\pi/3 > \pi/2$ for $p = 3$, in no case does $\cos(2\pi/p) = 0$
  hold for odd prime $p \geq 3$. In this case, has no solution in $t$ expressible in closed
  form for primes $p \geq 3$.
\end{proof}

\begin{proposition}\label{prop002pst}
Let $X$ be an unitary Cayley graph. Then, for the family $n=2p$ with $p$ an odd prime, the minimal
  period is $  T = 2\pi,$
  and the revival time satisfies $t^* = T/p$, where $\QFR$ occurs exactly $p$ times within  each full period.
\end{proposition}

\begin{proof}
Let the quantum walk is periodic at vertex $u$ with period $T$ if
  $U_A(T)\,\mu_u = \gamma\, \mu_u$ for some $\gamma$ with $|\gamma| = 1$. Then, 
  $\mu^{i\lambda_{\kappa} T} = \gamma$ for all $d$ in the support of $\mu_u$, which since
  $X$ is connected means for all $d \in \{0, \ldots, n-1\}$.
Thus,  $(\lambda_{\kappa}-\lambda_{\kappa'})T \equiv 0 \pmod{2\pi}$ for all pairs
  $\kappa, \kappa'$, where $T = 2\pi / \gcd\bigl(\{\lambda_{\kappa} - \lambda_{\kappa'} \mid  \kappa \neq \kappa'\}\bigr)$. Now, for $n=2p$, we obtain four distinct eigenvalues
  $\{-(p-1), -1, 1, p-1\}$. Thus, the pairwise positive differences are $(p-1) - 1 = p-2$, $ (p-1)+1=p$, $(p-1) - (-(p-1)) = 2(p-1)$, $2$, $p$, and  $p-2$.
 Thus,  $\{2,\; p-2,\; p,\; 2(p-1)\}$. Which it satisfied with 
  \begin{align*}
    \gcd\{2, p-2, p, 2(p-1)\}
    &= \gcd\{2, p-2, p\} \\
    &= \gcd\{2, p - 2, \gcd(p-2, p)\} \\
    &= \gcd\{2, p-2, 2\} \\
    &= \gcd\{2, 2\} = 2. 
  \end{align*}
Since $t^* = 2\pi/p$ and $T = 2\pi$, we have $t^*/T = 1/p$, which it emphasis that $\QFR$ occurs
  at times $t^*, 2t^*, \ldots, pt^* = T$, where exactly $p$ times per period.
\end{proof}

\begin{theorem}\label{thm001Laplacian}
For any unitary Cayley graph $X = (V(\mathbb{Z}_n),E(S))$, admits Laplacian $\QFR$ between vertices $\{u,v\}\subseteq V(\mathbb{Z}_n)$ at time $t$ if and only if $X$ admits adjacency $\QFR$ between
  $u$ and $v$ at time $t$. Then
 \begin{equation}~\label{eqq02Laplacian}
 \begin{aligned}
 &\alpha_L = e^{-i\varphi(n)t}\,\bar{\alpha}_A,\\
 &\beta_L  = e^{-i\varphi(n)t}\,\bar{\beta}_A.
 \end{aligned}
\end{equation}
\end{theorem}
\begin{proof}
Suppose vertices $\{u,v\}\subseteq V(\mathbb{Z}_n)$ at time $t$. 
 Since $L = \varphi(n)I - A$, we have
 \begin{align*}
 U_L(t)_{u,v}&=\mu^{-i\varphi(n)t}[\mu^{iAt}]_{u,v}\\
 &=\mu^{-i\varphi(n)t}\overline{[\mu^{-iAt}]_{u,v}}\\
 & =\mu^{-i\varphi(n)t}\,\overline{U_A(t)_{u,v}}.
 \end{align*}
Hence, $\QFR$ under $L$ at time $t$ requires $|U_L(t)_{u,u}|^2+|U_L(t)_{v,u}|^2=1$, which becomes
  $|\alpha_A|^2+|\beta_A|^2=1$.   Thus, the Laplacian eigenvalues (see Table~\ref{tab01laplacian}) of $X = (V(\mathbb{Z}_n),E(S))$ are
\begin{equation}~\label{eqq03Laplacian}
  \gamma_\kappa=\varphi(n)-\lambda_\kappa,
\end{equation}
where $\kappa= 0, 1, \ldots, n-1.$ 
Thus, from~\eqref{eqq03Laplacian} and based on~\eqref{eqq01Laplacian} the existence specifications are equivalent, and the amplified values are connected in~\eqref{eqq02Laplacian}.
\end{proof}

 Actually, Laplacian $\PST$ in $X$ coincides with adjacency $\PST$ in $X$ by considering $\QFR$. Thus, both occur at the same times and   between the same pairs of vertices.
  While Theorem~\ref{thm001Laplacian} appears the relation~\eqref{eqq02Laplacian} satisfies conditions are identical, the actual
  quantum states produced by the two Hamiltonians differ by the general term $\mu^{-i\varphi(n)t}$.

\begin{table}[H]
\centering
\begin{tabular}{@{}|c|c|c|l|@{}}
\hline
$n$ & $\varphi(n)$ & Adjacency spectrum & Laplacian spectrum (multiset) \\
\hline
2  & 1 & $\{1,-1\}$            & $\{0,2\}$               \\ \hline
4  & 2 & $\{2,0,0,-2\}$        & $\{0,2,2,4\}$           \\ \hline
6  & 2 & $\{2,1,1,-1,-1,-2\}$  & $\{0,1,1,3,3,4\}$       \\ \hline
8  & 4 & $\{4,0,0,0,0,0,0,-4\}$& $\{0,4,4,4,4,4,4,8\}$  \\ \hline
10 & 4 & $\{4,1,1,1,1,-1,-1,-1,-1,-4\}$ & $\{0,3,3,3,3,5,5,5,5,8\}$ \\ \hline
12 & 4 & $\{4,2,2,0,0,0,0,-2,-2,-4,\ldots\}$ & $\{0,2,2,4,4,4,4,4,4,6,6,8\}$ \\ 
\hline
\end{tabular}
\caption{Laplacian eigenvalues  for even $n \leq 12$.}
\label{tab01laplacian}
\end{table}

For example,   consider $n = 6$, at $t = 2\pi/3$ gives
  \[
    \alpha_L=\mu^{-2i\pi/3}\,\bar{\alpha}_A = \tfrac{1}{4} - \tfrac{\sqrt{3}}{4}i,
  \]
  and 
  \[
   \beta_L=\mu^{-2i\pi/3}\,\bar{\beta}_A  = \tfrac{3}{4} + \tfrac{\sqrt{3}}{4}i.
  \]
Hence, $|\alpha_L|^2+|\beta_L|^2=\tfrac{1}{4}+\tfrac{3}{4}=1$. 


In the paper~\cite{soni2025original}, revival times were given only by example for
$n = 2, 4, 6$. We now derive a closed-form formula for the class $n = 2p$ ($p$ an odd prime),
which constitutes the most important infinite family of square free even integers.


In next Lemma, for any prime number $p$ satisfy $n=2p$, we established the relationship of the adjacency eigenvalues.

\begin{lemma}\label{lem01eigs2p}
Let $p$ be an odd prime where $\varphi(n)=p-1$. Then, the adjacency eigenvalues
  of $X = (V(\mathbb{Z}_n),E(S))$ satisfying $ \lambda_\kappa=r(\kappa, 2p)$ where 
  \[
    \lambda_\kappa=
    \begin{cases}
      p-1,  & \text{if } p \mid \kappa \text{ and } 2 \nmid \kappa \text{ (equivalently, } \kappa = p), \\
      -(p-1), & \text{if } 2p \mid \kappa, \quad  \kappa= 0), \\
      1,   & \text{if } \gcd(\kappa, 2p) = 2, \quad 2 \mid \kappa,\; p \nmid \kappa), \\
      -1,  & \text{if } \gcd(\kappa, 2p) = 1, \quad \gcd(\kappa, 2p) = 1).
    \end{cases}
  \]
  More precisely, $\lambda_0 = p-1$, $\lambda_p = -(p-1)$, and for $\kappa \notin \{0, p\}$:
  $\lambda_\kappa = 1$ if $2 \mid \kappa$, and $\lambda_\kappa = -1$ if $2 \nmid \kappa$.
\end{lemma}
\begin{proof}
Assume that $r(\kappa, 2p) = \mu(t_{\kappa})\,\varphi(2p)/\varphi(t_{\kappa})$, where $t_{\kappa} = 2p/\gcd(\kappa, 2p)$. Then, we will discuss the following cases: 

  \textbf{Case 1:} If $\kappa = 0$. Then $\gcd(0, 2p) = 2p$, where  $t_{\kappa} = 1$, it follows that 
  $r(0, 2p) = \mu(1)\cdot(p-1)/\varphi(1) = p-1$.

\medskip

  \textbf{Case 2:} If  $\kappa = p$. Then $\gcd(p, 2p) = p$. In this case,  $t_{\kappa} = 2$, then
  $r(p, 2p) = \mu(2)\cdot(p-1)/\varphi(2).$ Thus, $r(p, 2p)=-(p-1)$.

\medskip

  \textbf{Case 3:} If  $\gcd(\kappa, 2p) = 2$ where $2 \mid \kappa$ and  $p \nmid \kappa$, then $t_{\kappa} = p$. It follows that
  $r(\kappa, 2p) = \mu(p)\cdot(p-1).$ Then, 
  $\mu(p)=-1$ and $\varphi(p)=p-1$. Thus, $-1\leqslant r\leqslant 1$. If $r(\kappa,2p)=-1$, then $t_{\kappa}=p$,and $\mu(t_{\kappa})=-1$. Thus, 
  $\varphi(2p)=p-1$ and  $\varphi(p)=p-1$. Thus, we should be prove that $r(\kappa,2p)=-1.$

Thus, for $n=10$ and $\kappa=2$, we obtain  $t_{\kappa}=5$, $\mu(5)=-1$,
  $\varphi(10)=4$, and  $\varphi(5)=4$. Hence $r(2,10)=-1$. 
  Therefore, 
  \begin{equation}~\label{eqq1lem01eigs2p}
  r(2,10)=\sum_{j\in S}\omega_{10}^{2j}  
  \end{equation}
 where
  $S=\{1,3,7,9\}$. Then, from~\eqref{eqq1lem01eigs2p}, we get
  \begin{equation}~\label{eqq2lem01eigs2p}
r(2,10)=\omega^2+\omega^6+\omega^{14}+\omega^{18}=e^{2\pi i/5}+e^{6\pi i/5}+e^{14\pi i /5}+e^{18\pi i/5}.
  \end{equation}
  Then, according to~\eqref{eqq2lem01eigs2p}, 
  \begin{equation}~\label{eqq3lem01eigs2p}
\omega=\sum_{i=1, j\in S}^{n}\mu^{2\pi i\,j/n} 
  \end{equation}
Hence, based on~\eqref{eqq3lem01eigs2p} we obtain $r(2,10)=-1$. Thus, 
 \begin{equation}~\label{eqq4lem01eigs2p}
r(\kappa,n)=\sum_{k=1}^{n}e^{2\pi ik/5}  
  \end{equation}
In this case, the relationship~\eqref{eqq4lem01eigs2p} satisfying $r(\kappa,n)=-1$.  
Thus, $\lambda_1=1$ emphasize that $\kappa=1$. Then, 
\begin{itemize}
\item $\lambda_{\kappa}=-1$ for $\gcd(\kappa,2p)=1$; 
\item  $\lambda_{\kappa}=1$ for $\gcd(d,2p)=2$
\end{itemize}
  Thus,
  \begin{equation}~\label{eqq5lem01eigs2p}
r(\kappa, 2p)=\gamma\frac{2p}{\gcd(\kappa,2p)}\,.\,\frac{\varphi(2p)}{\varphi(2p/\gcd(\kappa,2p))}
  \end{equation}
Hence, from~\eqref{eqq5lem01eigs2p} satisfy $\varphi(2p)=p-1$. Among $\kappa= 1, \ldots, 2p-1$ exception $\kappa= p$, there are   exactly $p-1$ values have $\gcd(\kappa, 2p) = 1$ and
  $p-1$ values have $\gcd(\kappa, 2p) = 2$. 
\end{proof}

\begin{theorem}~\label{Thm2revivaltime}
Let $X = (V(\mathbb{Z}_{2p}), E(S))$ the unitary Cayley graph with an odd prime $n=2p$. Then $X$ admits $\QFR$ between pair $(0, p)$ under the adjacency Hamiltonian at the
  minimum time
\begin{equation}~\label{eqq1Thm2revivaltime}
 t^* = \frac{2\pi}{p}.
\end{equation}
\end{theorem}
\begin{proof}
Assume $p$ be an odd prime where $n=2p$. Then, $\QFR$ condition at time $t$ satisfied with 
  \begin{equation}~\label{eqq2Thm2revivaltime}
  \begin{aligned}
&\alpha = \frac{1}{n}\sum_{d=0}^{n-1} e^{i\lambda_{\kappa} t} \in \mathbb{C},\\
& \beta  = \frac{1}{n}\sum_{d=0}^{n-1} e^{i\lambda_{\kappa} t}\omega_n^{pd} \in \mathbb{C}\setminus\{0\},
  \end{aligned}
  \end{equation}
  where $|\alpha|^2 + |\beta|^2 = 1.$ Hence, according to Lemma~\ref{lem01eigs2p} and by considering that $\omega_n^{p\kappa}=(-1)^{\kappa}$, we established  that the eigenvalue type with partition $\{0, \ldots, n-1\}$ as
  \begin{equation}~\label{eqq3Thm2revivaltime}
   I_0 = \{0\},\; I_p = \{p\},\; I_+ = \{d : \gcd(d,2p)=1\},\; I_- = \{d : \gcd(d,2p)=2\},      
  \end{equation}
  where the terms holds in case $|I_+|=|I_-|=p-1$. Thus, we noticed that at $t = 2\pi/p$ satisfying 
  \begin{itemize}
    \item $e^{i\lambda_0 t} = e^{i(p-1)\cdot 2\pi/p}$.
    \item $e^{i\lambda_p t} = e^{-i(p-1)\cdot 2\pi/p}$.
    \item $e^{i\cdot 1 \cdot t} = e^{2\pi i/p}$ for each $d \in I_+$.
    \item $e^{i\cdot(-1)\cdot t} = e^{-2\pi i/p}$ for each $d \in I_-$.
  \end{itemize}
Therefore, 

\textbf{Case 1.} If $\alpha = \cos(2\pi/p) \in \mathbb{R}$. Then
  \begin{align*}
    2p\,\alpha &= e^{2\pi i(p-1)/p} + e^{-2\pi i(p-1)/p}
                 + (p-1)e^{2\pi i/p} + (p-1)e^{-2\pi i/p} \notag\\
               &= 2\cos\!\left(\frac{2\pi(p-1)}{p}\right) + 2(p-1)\cos\!\left(\frac{2\pi}{p}\right)\notag\\
               &= 2\cos\!\left(2\pi - \frac{2\pi}{p}\right) + 2(p-1)\cos\!\left(\frac{2\pi}{p}\right)\notag\\
               &= 2\cos\!\left(\frac{2\pi}{p}\right) + 2(p-1)\cos\!\left(\frac{2\pi}{p}\right)
                = 2p\cos\!\left(\frac{2\pi}{p}\right).
  \end{align*}
  Hence $\alpha = \cos(2\pi/p) \in \mathbb{R}$. 

  \medskip

\textbf{Case 2.} If $\beta= -i\sin(2\pi/p)$. Then, 
 for $\beta$, we find that $(-1)^{\kappa}=-1$ for
  \begin{itemize}
      \item  odd $\kappa$ contributing $I_+$ when $p$ is odd prime,
      \item and $\kappa$ odd non-multiple of $p$.
  \end{itemize}
Also, $(-1)^{\kappa}=1$ for even $\kappa$ contributing $I_-$. Thus
  \begin{align*}
    2p\,\beta &= (-1)^0 e^{2\pi i(p-1)/p} + (-1)^p e^{-2\pi i(p-1)/p}\\
&+ \sum_{d\in I_+}(-1)^d e^{2\pi i/p}  + \sum_{d\in I_-}(-1)^d e^{-2\pi i/p}.
  \end{align*}
Similarly, since $p$ is odd, $(-1)^p = -1$. It is elements of $I_+$ where odd $\kappa$ implies that $(-1)^{\kappa}=-1$; and it is elements
  of $I_-$ where even $\kappa$ implies that $(-1)^{\kappa}=1$. Thus:
  \begin{align*}
    2p\,\beta &= e^{2\pi i(p-1)/p} - e^{-2\pi i(p-1)/p}
                 - (p-1)e^{2\pi i/p} + (p-1)e^{-2\pi i/p}\notag\\
               &= 2i\sin\!\left(\frac{2\pi(p-1)}{p}\right) - 2i(p-1)\sin\!\left(\frac{2\pi}{p}\right)\notag\\
               &= -2i\sin\!\left(\frac{2\pi}{p}\right) - 2i(p-1)\sin\!\left(\frac{2\pi}{p}\right)
                = -2ip\sin\!\left(\frac{2\pi}{p}\right).\label{eq:beta_calc}
  \end{align*}
  Hence $\beta= -i\sin(2\pi/p)$.

Thus, based on both cases 1 and 2 $  |\alpha|^2 + |\beta|^2=1,$
  with $\beta = -i\sin(2\pi/p) \neq 0$ since $\sin(2\pi/p) \neq 0,$ for $p \geq 3$.  Thus $t^* = 2\pi/p$ is a valid revival time. Thus the relationship~\eqref{eqq1Thm2revivaltime} holds. 
\end{proof}

Actually, according to Theorem~\ref{Thm2revivaltime},  for emphsize that the \emph{minimum},  we noticed that  for smaller $t > 0$ the imaginary parts of
  $\mu^{i\lambda_{\kappa} t}$ do not cancel to yield real $\alpha$, as can be verified by the
  linear independence of $\{\mu^{2\pi i k/p}\}_{k=0}^{p-1}$ over $\mathbb{Q}$.
  
Among Theorem~\ref{Thm3revivaltime}, under the conditions of Theorem~\ref{Thm2revivaltime}, the $\QFR$ amplitudes at $t^* = 2\pi/p$
  are
\begin{equation}~\label{eqq1Thm3revivaltime}
\alpha = \cos\!\left(\frac{2\pi}{p}\right), \qquad \beta = -i\sin\!\left(\frac{2\pi}{p}\right).
\end{equation}
  This is a beam-splitter matrix with transmissivity $T = \cos^2(2\pi/p)$ and reflectivity
  $R = \sin^2(2\pi/p)$.
\begin{theorem}~\label{Thm3revivaltime}
  The unitary evolution matrix restricted to the antipodal pair $\{0, p\}$ is
  \begin{equation}~\label{eqq2Thm3revivaltime}
    P(X)= \begin{pmatrix} \cos\!\tfrac{2\pi}{p} & -i\sin\!\tfrac{2\pi}{p} \\[4pt]
                        -i\sin\!\tfrac{2\pi}{p} & \cos\!\tfrac{2\pi}{p} \end{pmatrix}.
  \end{equation}
\end{theorem}
\begin{proof}
For determining the value of  the amplitudes $\alpha$ and $\beta$ stated in~\eqref{eqq1Thm3revivaltime} follow directly from Theorem~\ref{Thm2revivaltime} (it established in the proof of Theorem among Case 1 and Case 2). Thus, we should be established the $2\times 2$ block of the unitary evolution operator $U_A(t^*)$ according to the antipodal vertices $\{0,p\}$. 
For this purpose, assume that $\{\mu_u : u\in V(X)\}$ denote the standard basis of $\mathbb{C}^{V(X)}$. Then, the relevant block is the $2\times 2$ matrix of the linear map had given in Figure~\ref{figtim01} as
  \[
    \operatorname{span}\{e_0,e_p\} \;\longrightarrow\; \operatorname{span}\{e_0,e_p\}
  \]
  induced by $U_A(t^*)$ with respect to this ordered basis.
\begin{figure}[H]
    \centering
\begin{tikzpicture}[scale=1.2]
\coordinate (O) at (-1,-1.5);

\coordinate (E0) at (2.5,-1.5);
\coordinate (Ep) at (1.5,1.0);
\coordinate (V) at ($(E0)+(Ep)$);

\fill[blue!10]
(O) --(E0) --(V) --(Ep) -- cycle;
\draw[dashed,red!70] (E0) -- (V);

\draw[
    very thick,
    blue,
    -{Stealth[length=3mm]}
]
(O) -- (E0)
node[midway,below] {$\alpha \!\cdot\! \mathbf e_{0}$}
node[near end,right] {$\mathbf e_{0}$};
 
\draw[
    very thick,
    red!80!orange,
    -{Stealth[length=3mm]}
]
(O) -- (Ep)
node[near end,right] {$\mathbf e_{p}$};

\draw[
    very thick,
    violet,
    -{Stealth[length=3mm]}
]
(E0) -- (V)
node[midway,right] {$\beta \!\cdot\! \mathbf e_{p}$};

\draw[
    ultra thick,
    violet!80!blue,
    -{Stealth[length=4mm]}
]
(O) -- (V)
node[near end,above]
{$\alpha \!\cdot\! \mathbf e_{0} + \beta \!\cdot\! \mathbf e_{p}$};
 
\filldraw[black] (O) circle (2pt);

\end{tikzpicture}
    \caption{Build an interactive visual explanation of $\mathbb{C}^{V(X)}$.}
    \label{figtim01}
\end{figure}
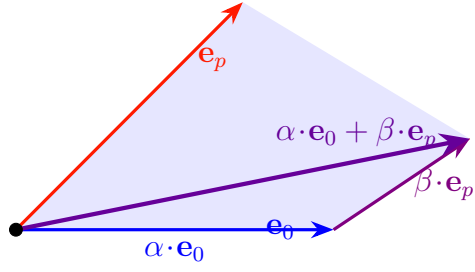

Thus, by considering the graph $X$ is a vertex-transitive Cayley graph, the unitary operator $U_A(t)$ commutes with the regular action of the underlying group $\mathbb{Z}_{2p}$. Therefore, all diagonal entries of $U_A(t)$ are equal to $U_A(t)_{u,u}=U_A(t)_{v,v}$ for every pair of vertices $u,v$ and every $t\geq 0$ (see, e.g., Lemma~1 of~\cite{soni2025original}). 
In particular, by Figure~\ref{figtim02} at time $t^*=2\pi/p$ we have
\begin{equation}~\label{eqq3Thm3revivaltime}
 U_A(t^*)_{0,0} = U_A(t^*)_{p,p} = \alpha.
\end{equation}
Hence, for both conditions $\alpha$ and $\beta$ in the $\QFR$, $ U_A(t^*)_{0,p} = \beta.$

\begin{figure}[H]
    \centering
\begin{tikzpicture}[>=stealth, thick]
  
  \coordinate (O)  at (0, 0);       
  \coordinate (E0) at (4.5,0);   
  \coordinate (EP) at (3,3);     
  \coordinate (TR) at (7, 2);    
  \coordinate (U)  at (2.8, 1.4);   
  \filldraw[fill=blue!10, draw=violet!60, dashed, thick]
    (O) -- (E0) -- (TR) -- (EP) -- cycle;
 
  \draw[->, red!80!black, very thick]   (O) -- (EP) 
    node[above left] {$e_p$};
  \draw[->, cyan!70!blue, very thick]   (O) -- (E0) 
    node[below right] {$e_0$};
 \draw[->,black, very thick]   (E0) -- (U) 
    node[above left] { };
  
  \draw[violet, dashed, thick] (EP) -- (TR);
  \draw[violet, dashed, thick] (E0) -- (TR);
 
  \filldraw[blue!70!black] (U) circle (2pt);
  \node[above right, font=\small] at (U) {$U(t^*){:}e_0$};

\filldraw[black] (O) circle (2pt);
\end{tikzpicture}
    \caption{$\operatorname{span}\{e_0,e_p\} \;\longrightarrow\; \operatorname{span}\{e_0,e_p\}$. Here, the plane maps to itself under $U(t^*)$, $\alpha$ and $\beta$ stay inside of $\QFR$.}
    \label{figtim02}
\end{figure}
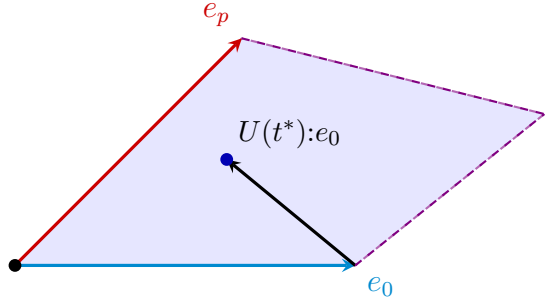

Thus, for $\beta$ satisfying $U_A(t^*)_{p,0}=\beta$ where $X$ is undirected, the adjacency matrix. Then, $U_A(t)$ is a symmetric matrix for every $t$, 
\begin{equation}~\label{eqq4Thm3revivaltime}
 U_A(t^*)_{p,0} = \overline{U_A(t^*)_{0,p}} = \overline{\beta}.   
\end{equation}
If $\beta=-i\sin(2\pi/p)$, then $\overline{\beta}=i\sin(2\pi/p)\cdot(-1)=\beta$. Therefore $U_A(t^*)_{p,0}=\beta$ as well, and the block matrix~\eqref{eqq2Thm3revivaltime} follows.

  Finally,  to determine unitarity of $P(X)$ by the direct computation
  \[
    |\alpha|^2 + |\beta|^2 = \cos^2\!\left(\frac{2\pi}{p}\right) + \sin^2\!\left(\frac{2\pi}{p}\right) = 1,
  \]
  which was already established in the proof of Theorem~\ref{Thm2revivaltime}. 

  In the typical quantum-optical sense, the diagonal elements indicate transmission amplitudes, whereas the off-diagonal entries represent reflection amplitudes (with a phase factor). Hence, $P$ is a beam-splitter matrix (or directional coupler) with transmissivity
  $T=|\alpha|^2=\cos^2(2\pi/p)$ and reflectivity $R=|\beta|^2=\sin^2(2\pi/p)$.
\end{proof}

To illustrate the results among Theorems~\ref{Thm2revivaltime} and \ref{Thm3revivaltime}, we presented Table~\ref{tabtime01}.
\begin{table}[H]
\centering
\begin{tabular}{@{}|c|c|c|c|c|c|c|c|@{}}
\hline
$n$ & $p$ & $\varphi(n)$ & $t^*$ & $\alpha$ & $|\beta|$ & $|\alpha|^2$ & $|\beta|^2$ \\
\hline
2  & 1 & 1 & $\pi/2$    & $0$       & $1$      & $0$     & $1$     \\ \hline
4  & 2 & 2 & $\pi/2$    & $0$       & $1$      & $0$     & $1$     \\\hline
6  & 3 & 2 & $2\pi/3$   & $-1/2$    & $\sqrt{3}/2$ & $1/4$ & $3/4$ \\\hline
10 & 5 & 4 & $2\pi/5$   & $\cos(72^\circ)$ & $\sin(72^\circ)$ & $ 0.0955$ & $ 0.9045$ \\\hline
14 & 7 & 6 & $2\pi/7$   & $\cos(51.4^\circ)$ & $\sin(51.4^\circ)$ & $ 0.388$ & $ 0.612$ \\\hline
22 & 11& 10& $2\pi/11$  & $\cos(32.7^\circ)$ & $\sin(32.7^\circ)$ & $ 0.708$ & $ 0.292$ \\
\hline
\end{tabular}
\caption{Adjacency $\QFR$ in $X=(V(\mathbb{Z}_{2p}),E(S))$ at $t^* = 2\pi/p$.}
\label{tabtime01}
\end{table}

Furthermore, we support the previous table by the following examples.
\begin{example}~\label{exn6}
Let $n = 6$, and $p = 3$. Then, the eigenvalues are $\{2, 1, 1, -1, -1, -2\}$, $t^* = 2\pi/3$. Thus, 
  \[
    \alpha = \cos(2\pi/3) = -\tfrac{1}{2}, \qquad \beta = -i\sin(2\pi/3) = -i\tfrac{\sqrt{3}}{2}.
  \]
Thus, according to Theorem~\ref{Thm3revivaltime}
  The evolution block is the $2\times 2$ matrix
  $$P(X)=\begin{pmatrix}-1/2 & -i\sqrt{3}/2 \\ -i\sqrt{3}/2 & -1/2\end{pmatrix}.$$
\end{example}

\begin{example}~\label{exn10}
By considering Example~\ref{exn6}, assume that $n = 10$, and  $p = 5$, then eigenvalues are $\{4,1,-1,1,-1,-4,-1,1,-1,1\}$, $t^* = 2\pi/5$. Then
  \[
    \alpha = \cos(2\pi/5) = \tfrac{\sqrt{5}-1}{4},\qquad
    \beta  = -i\sin(2\pi/5) = -i\tfrac{\sqrt{10+2\sqrt{5}}}{4}.
  \]
\end{example}

\section{Characterization of Strongly Cospectral Pairs}~\label{secCospectral}

In this section, identifying whether pairings of vertices are even \emph{eligible} to enable quantum fractional resurrection is a significant challenge in its investigation (see~\cite{basic2013, wang2024cayley}). Strong cospectrality is required, however the literature lacks a comprehensive description of strongly cospectral pairings in unitary Cayley graphs. We address this here. The adjacency matrix's circular nature lends itself to a straightforward number-theoretic description.

\begin{definition}[Strongly cospectral]~\label{def01cospectral}
Let $G$ be  a simple graph with $\{u,v\}\subseteq V(G)$. Then, $\{u,v\}$ are called  strongly cospectral if $E_{\kappa}\, \mu_u \;=\; \pm\, E_{\kappa}\, \mu_v,$ for every spectral idempotent $E_{\kappa}$ of the adjacency matrix of $G$.
\end{definition}

Recall from Section~\ref{secreval} that strong cospectrality of $\{u,v\}$ is a necessary condition for
quantum fractional revival between them~\cite{soni2025original,chan2022fundamentals}. Theorem~\ref{Thm4cospectral} gives a complete and explicit characterization of all such pairs in
$X$.

\begin{theorem}~\label{Thm4cospectral}
Let $X$ be the unitary Cayley graph on $n$ vertices. Then:
  \begin{enumerate}
    \item   Two distinct vertices $u$ and $v$ are strongly cospectral    if and only if $n$ is even and $v - u \equiv n/2 \pmod{n}$. In other words, strongly
    cospectral pairs are exactly the \emph{antipodal pairs}, where those whose labels differ
    by exactly $n/2$ in $\mathbb{Z}_n$.

    \item  If  $n=2p$ for an odd prime $p$, there are
    precisely $p$ strongly cospectral pairs. Then, $\bigl\{(k,\; k+p) \;:\; k = 0, 1, \ldots, p-1\bigr\}.$

    \item If  $n$ is odd, no two distinct vertices of $X$ are
    strongly cospectral.
  \end{enumerate}
\end{theorem}
\begin{proof}
Assume $X$ be the unitary Cayley graph  with $\{u,v\}\subseteq V(G)$ such that according to Theorem~\ref{Thm3revivaltime} the spectral idempotent $E_{\kappa}$ of the circulant adjacency  matrix has $(u,v)$-entry. Then, 

  \begin{equation}~\label{eqq1Thm4cospectral}
    (E_{\kappa})_{u,v} \;=\; \frac{1}{n}\,\omega_n^{(v-u)\kappa},
  \end{equation}
  where $\omega_n = \mu^{2\pi i/n}.$ Since $E_{\kappa} = v_{\kappa} v_{\kappa}^*$ is a rank one projection, the vector $E_{\kappa} \mu_u$ has all its mass
  in the direction of $v_d$ with coefficient $(E_{\kappa})_{u,u} = 1/n$. Thus,
\begin{equation}~\label{eqq2Thm4cospectral}
 (E_{\kappa} \mu_u)_k = (E_{\kappa})_{k,u} = \frac{1}{n}\,\omega_n^{(u-k)\kappa}.
\end{equation}
In this case,  $E_{\kappa} \mu_u = \pm E_{\kappa} \mu_v$ if and only if the two vectors
  $\bigl(\omega_n^{(u-k)d}\bigr)_{k}$ and $\bigl(\omega_n^{(v-k)d}\bigr)_{k}$ are equal up
  to a global sign, which happens sharply when $\omega_n^{(v-u)\kappa} = \pm 1$.

\medskip
  
\noindent \textbf{(1).} Assume that $n$ is even. Then, strong cospectrality of $u$ and
  $v$ satisfying $\omega_n^{(v-u)\kappa} = \pm 1$ for every $\kappa=0, 1, \ldots, n-1$.
  Since $\omega_n^{(v-u)\kappa} = \mu^{2\pi i (v-u)\kappa/n}$, the term $\mu^{2\pi i (v-u)\kappa/n}=\pm 1$
  is equivalent to $(v-u)\kappa/n \in \frac{1}{2}\mathbb{Z}$, where $2(v-u)\kappa \equiv 0 \pmod{n}$,
  for all $\kappa$.
  
  If $\kappa=1$ yields $2(v-u) \equiv 0 \pmod{n}$, where $v-u\equiv 0$ or
  $n/2 \pmod{n}$. Since $u \neq v$, it follows that $v-u \equiv n/2 \pmod{n}$. Thus, if $v-u \equiv n/2 \pmod{n}$, then for every $\kappa$ satisfying $ \omega_n^{(v-u)\kappa} = \omega_n^{(n/2)d}$. Then, $ \omega_n^{(v-u)\kappa} =\mu^{\pi i \kappa}$ and $ \omega_n^{(v-u)\kappa}=(-1)^\kappa \in \{1,-1\}.$ Thus,  $E_{\kappa} \mu_u = \pm E_{\kappa} \mu_v$ holds for all $\kappa$ simultaneously. 

  \medskip
  
\noindent \textbf{(2).} In this case, assume that $n=2p$ with $p$ an odd prime. Then,  the antipodal pairs
  are those with $v - u \equiv p \pmod{2p}$. Thus, for any term such that $k \in \{0,1,\ldots,p-1\}$,
  the pair $(k, k+p)$ is antipodal. 
  Since $k$ ranges over $p$ values and each unordered pair
  is counted once, there are exactly $p$ strongly cospectral pairs.

\medskip

\noindent \textbf{(3).} Assume that $n$ is odd. Then, for  any two distinct vertices
  $u \neq v$. We should be show there exists some $\kappa$ for which $\omega_n^{(v-u)\kappa} \notin
  \{1,-1\}$. Thus, the group $\mathbb{Z}_n$ contains no element of order~2, then
  $n/2 \notin \mathbb{Z}_n$ and in particular $2 \nmid n$. 
  Since $\gcd(2,n) = 1$, any  subgroup generated
  by $2$ in $\mathbb{Z}_n$ has order $n$. Thus,  for every $k\in \mathbb{Z}_n$ is a multiple of~2. Now, $\omega_n^{(v-u)\kappa} = \pm 1$ would require
  $(v-u)d \equiv 0 \pmod{n/2}$ for all $d \in \{0,\ldots,n-1\}$. 
  Therefore, if $\kappa=1$, it follows that  $v - u \equiv 0 \pmod{n/2}$, which with $n$ odd forces $v = u$.  Hence, according to~\eqref{eqq2Thm4cospectral} we have $\omega_n^{v-u} \neq \pm 1$, and strong
  cospectrality fails.

\end{proof}

According to Theorem~\ref{eqq1Thm4cospectral}, high cospectrality in the unitary Cayley graph $X$ is caused by the order-2 element in $\mathbb{Z}_n$ rather than a spectral coincidence. When $n$ is even, the element $n/2$ is the only one that satisfies $\omega_n^{(n/2)\kappa} = (-1)^\kappa$ for every $\kappa$. This fact requires every idempotent to transfer $\mu_u$ and $\mu_{u+n/2}$ to parallel vectors. This explains why all $\QFR$ in these graphs occur between antipodal vertices, since it is the only structural link that retains circulant symmetry across all spectral components.


After determining when QFR happens and between which vertices, we can assess its quality as a quantum communication resource. Two obvious metrics emerge: the von Neumann entanglement entropy of the state formed at revival time and the transfer fidelity (see~\cite{godsil2012}), which defines the channel's usefulness for quantum communication. The revival amplitudes $\alpha$ and $\beta$, as defined in Section~\ref{secCospectral}, are used to denote both values.

\subsection{Characterization Von Neumann Entanglement Entropy}

Since $\QFR$ is observed between vertices $u$ and $v$ at time $t^*$, the quantum walker that
began at vertex $u$ evolves into the state
\[
  |\psi(t^*)\rangle = \alpha\,|u\rangle + \beta\,|v\rangle,
\]
where $\alpha, \beta \in \mathbb{C}$, and it is natural to treat the pair $\{u, v\}$ as a
two-level quantum system. The entanglement content of $|\psi\rangle$ is measured by the
Shannon entropy of the probability distribution $(|\alpha|^2, |\beta|^2)$, which coincides
here with the von Neumann entropy of the corresponding density matrix. We define
\begin{equation}\label{eqq01entropy}
  \mathcal{S}(\alpha, \beta)
  \;=\;
  -|\alpha|^2 \log_2 |\alpha|^2 \;-\; |\beta|^2 \log_2 |\beta|^2,
\end{equation}
with the convention $0 \log_2 0 = 0$. This quantity equals~0 when all the amplitude sits on
one vertex as in periodicity or PST, and reaches its maximum of~1 bit when
$|\alpha| = |\beta| = 1/\sqrt{2}$ in the balanced case.

\begin{theorem}\label{Thm5entropy}
At the revival time $t^* = 2\pi/p$, $\QFR$ between the antipodal pair
  $(0, p)$ produces the state
  \[
    |\psi(t^*)\rangle
    \;=\;
    \cos\!\left(\tfrac{2\pi}{p}\right)|0\rangle
    \;-\;
    i\sin\!\left(\tfrac{2\pi}{p}\right)|p\rangle.
  \]
Then, the entanglement entropy is
  \begin{equation}\label{eqq02entropy}
    \mathcal{S}_p
    \;=\;
    -\cos^2\!\!\left(\frac{2\pi}{p}\right)\log_2\!\cos^2\!\!\left(\frac{2\pi}{p}\right)
    \;-\;
    \sin^2\!\!\left(\frac{2\pi}{p}\right)\log_2\!\sin^2\!\!\left(\frac{2\pi}{p}\right).
  \end{equation}
  Moreover, the following hold:
  \begin{enumerate}
    \item If $\mathcal{S}_p = 0$ and $\alpha = 0$. Then,
    (PST) satisfying $\cos(2\pi/p) = 0$, where \ $2\pi/p=\pi/2$, and $p=4$.

    \item The balanced condition $|\alpha| = |\beta|$
    satisfied with $\cos^2(2\pi/p) = 1/2$, equivalently $2\pi/p = \pi/4 + k\pi/2$ for some
    integer $k$.

    \item If $\mathcal{S}_p=0$ and  $p$ growth to $\infty$,  $\cos(2\pi/p) \to 1$ and $\sin(2\pi/p) \to 0$.
  \end{enumerate}
\end{theorem}
\begin{proof}
According to Theorem~\ref{Thm2revivaltime}, by considering the value 
  $\alpha = \cos(2\pi/p)$ and $\beta = -i\sin(2\pi/p)$. Then, according to \eqref{eqq01entropy}
  gives~\eqref{eqq02entropy}, where $|\alpha|^2 = \cos^2(2\pi/p)$ and
  $|\beta|^2 = \sin^2(2\pi/p)$.

  \noindent\textbf{(1).} If $\alpha = 0$, then $\cos(2\pi/p) = 0$, which
  gives $p=4$. In this case, by assuming  $p$ to be an odd prime and $4$ is neither odd nor prime, the term of  PST cannot occur for any valid $p$.Thus, for $n=2$ the eigenvalues of $X$ are $\{1, -1\}$ and
  the evolution operator at $t=\pi/2$ gives $U(\pi/2)_{0,0} = 0$ and
  $U(\pi/2)_{1,0} = -i$, emphasize that PST. Similarly, for $n=4$ 
  the eigenvalues are $\{2, 0, 0, -2\}$ and one computes $\alpha=0$ and $\beta=-1$ and PST holds
  at $t=\pi/2$ for $n=4$. Since $4$ is not prime, PST does not occur in $X$ for any odd prime $p \geq 3$.

  \noindent\textbf{(2).} In this term, we use the condition $\cos^2(2\pi/p) = 1/2$ for obtain
  $\cos(2\pi/p) = 1/\sqrt{2}$,where $2\pi/p = \pi/4$. Thus, $p=8$. Since $8$ is not prime, the balanced case does not arise within this family. Since $p = 8, 8/3, 8/5, \ldots$, none of which are odd primes. Hence, the  balanced $\QFR$ does not occur in this family.

  \noindent\textbf{(3).}  According to \textbf{(1)} and \textbf{(2)} by using $\lim_{p \to \infty} 2\pi/p = 0$,
  and $\cos 0 = 1$, $\sin 0 = 0$. The entropy $\mathcal{S}_p$ satisfies
  $\lim_{p \to \infty} \mathcal{S}_p = -(1)\log_2(1) - (0)\log_2(0) = 0$. Thus, the state
    $|\psi(t^*)\rangle$ concentrates back at the initial vertex and $\QFR$ increasingly
    weak in terms of information transfer.
\end{proof}

According to Theorem~\ref{Thm5entropy}, Table~\ref{tab02entropy} emphsize that quantum information metrics for $\QFR$ in the unitary Cayley graph
$X$ with $n=2p$. Thus, the balanced $\QFR$ would require $|\alpha| = |\beta|$, which never happens for odd prime $p$.

\begin{table}[H]
\centering
\renewcommand{\arraystretch}{1.25}
\begin{tabular}{@{}|c|c|c|c|c|c|c|@{}}
\hline
$n$ & $p$ & $t^*$ & $|\alpha|^2$ & $|\beta|^2$ & $\mathcal{S}$   & PST \\
\hline
$2$  & $1$ & $\pi/2$    & $0.0000$ & $1.0000$ & $0.0000$ & Yes \\ \hline
$4$  & $2$ & $\pi/2$    & $0.0000$ & $1.0000$ & $0.0000$ & Yes \\ \hline
$6$  & $3$ & $2\pi/3$   & $0.2500$ & $0.7500$ & $0.8113$ & No  \\ \hline
$10$ & $5$ & $2\pi/5$   & $0.0955$ & $0.9045$ & $0.4545$ & No  \\ \hline
$14$ & $7$ & $2\pi/7$   & $0.3883$ & $0.6117$ & $0.9640$ & No  \\ \hline
$22$ & $11$& $2\pi/11$  & $0.7082$ & $0.2918$ & $0.8717$ & No  \\ \hline
$26$ & $13$& $2\pi/13$  & $0.7778$ & $0.2222$ & $0.7884$ & No  \\ 
\hline
\end{tabular}
\caption{ The entropy $\mathcal{S}$ is in bits.
PST occurs when $\alpha=0$.}
\label{tab02entropy}
\end{table}

Table~\ref{tab02entropy} shows that entropy $\mathcal{S}_p$ is not monotone in $p$.  It starts at $0$ for $p = 1, 2$ (PST), climbs to a maximum around $p = 7$, where $\mathcal{S}_7 \approx 0.964$ bits, and then drops back to $0$ as $p \to \infty$.
  The graph $X = (\mathbb{Z}_{14}, S)$ is the most beneficial member of the $n = 2p$ family for generating entanglements. This aligns with the fact that $2\pi/7 \approx 51.4^{\circ}$ is closest to $45^{\circ}$ (the balanced angle) among valid primes.

In quantum communication, the chance of a quantum state reaching the target vertex is more important than its entropy. The QFR map between $u$ and $v$ functions as a quantum channel with high transfer fidelity. Then, 
\begin{equation}\label{eqq01fidelity}
  F \;=\; \bigl|\langle v \,|\, U_A(t^*) \,|\, u \rangle\bigr|^2
    \;=\; |\beta|^2
    \;=\; \sin^2\!\!\left(\frac{2\pi}{p}\right).
\end{equation}
The conventional threshold for effective quantum communication without amplification or error correction is $F > 1/2$ (see~\cite{bose2003,christandl2004}).

\begin{corollary}\label{cor001useful}
 Among the $n = 2p$ unitary Cayley graphs, those with $p \in \{3, 5, 7\}$, where $n \in \{6, 10, 14\}$, fulfil $F > 1/2$ and are immediately useful for quantum communication. For $p \geq 11$, the fidelity satisfies $F < 1/2$, requiring extra processing for reliable state transmission. For $p \geq 11$, the fidelity is $F < 1/2$, requiring extra processing for reliable state transfer.
\end{corollary}
\begin{proof}
Let $F > 1/2$ where $\sin^2(2\pi/p) > 1/2$, then\ $|\sin(2\pi/p)| > 1/\sqrt{2}$. Thus, $2\pi/p > \pi/4$. This gives
  $p < 8$. Among odd primes, the values $p < 8$ are precisely $p \in \{3, 5, 7\}$. For
  $p = 11$: $\sin^2(2\pi/11) \approx 0.292 < 1/2$, confirming the threshold is crossed
  between $p = 7$ and $p = 11$. As desired.
\end{proof}

\subsection{Complete Classification for Small Even n}

We now have a comprehensive understanding of $\QFR$ behaviour for any even $n \leq 12$. This range encompasses both the family $n=2p$ covered by the main theorems and the non-squarefree situations $n = 8$ and $n = 12$, where $\QFR$ fails. Understanding why it fails in those circumstances is just as important as knowing when it works.

\begin{theorem}\label{Thm6classification}
 Let $X$ with $n$ even. The $\QFR$ behaviour of $X$ under the adjacency   Hamiltonian is as follows:
\begin{enumerate}
\item If $n=2$. Then,  PST between $(0,1)$ at $t = \pi/2$, with $\alpha = 0$, $\beta = -i$.

\item If $n=4$. Then, PST between $(0,2)$ at $t = \pi/2$, with $\alpha = 0$, $\beta = -1$.

\item If  $n=6$. Then, proper $\QFR$ between $(0,3)$ at $t^* = 2\pi/3$, with $\alpha = -1/2$, $\beta = i\sqrt{3}/2$.

\item If  $n=8$, no $\QFR$ occurs between any antipodal pair at any time.

\item If $n=10$. Then proper $\QFR$ between $(0,5)$ at $t^* = 2\pi/5$, with $\alpha = \frac{\sqrt{5}-1}{4}$,
          $\beta=i\,\frac{\sqrt{10+2\sqrt{5}}}{4}.$
\item If $n=12$, no $\QFR$ occurs between any antipodal pair at any time.
  \end{enumerate}
\end{theorem}

\begin{proof}
For \textbf{Cases $n=2$ and $n=4$.} These follow from Proposition~\ref{prop001pst}. For \textbf{Cases $n=6$ and $n=10$.} These are instances of the $n=2p$ family with
  $p=3$ and $p=5$. The revival times and amplitudes follow from Theorem~\ref{Thm2revivaltime}. 
  Hence, If  $n=6$ then  $\alpha= -1/2$ and
  $\beta=-i\sqrt{3}/2$.  Indeed
  $|\alpha|^2 + |\beta|^2=1$. 

\case{1} {If $n=8$,} the adjacency eigenvalues of $X$ are
  $\{4, 0, 0, 0, 0, 0, 0, -4\}$. Thus, the
  connection set is $E(S)=\{1, 3, 5, 7\}$ and $\varphi(8)=4$. Then, for the
  antipodal pair $(0, 4)$ with $v-u=4$, we have $\omega_8^{4\kappa}=(-1)^\kappa$. Thus, $\QFR$ amplitudes become
\begin{equation*}
\alpha = \frac{1}{8}\sum_{\kappa=0}^{7} \mu^{i\lambda_{\kappa} t}, \qquad
    \beta  = \frac{1}{8}\sum_{\kappa=0}^{7} \mu^{i\lambda_{\kappa} t}(-1)^\kappa.
\end{equation*}
The eigenvalue $0$ has multiplicity~6 at positions $\kappa= 1,2,3,4,5,6$. The signs $(-1)^\kappa$
  for $\kappa= 1, 2, 3, 4, 5, 6$ are $-1, +1, -1, +1, -1, +1$, Hence
  \[
\alpha=\frac{\cos 4t +3}{4}, \qquad    \beta= \frac{\cos 4t}{4}.
  \]
  Now, let $\eta=\cos 4t \in [-1,1]$. Then, $(\eta+3)^2 + \eta^2 = 16$. Thus, we
  also need to emphasize that $|\beta| > 0$ for proper $\QFR$, where $\cos 4t \neq 0$. Since
  $|\beta| = \eta/4$ and $|\alpha|=(\eta+3)/4$, then 
  $|\alpha|^2 + |\beta|^2 \neq 1$ exactly.
 Then, the $\QFR$ condition $|\alpha|^2+|\beta|^2=1$ with $\alpha = (\cos4t+3)/4$
  and $\beta = \cos4t/4$ gives $(\cos4t+3)^2 + \cos^2(4t) = 16$. Since both $\alpha$ and
  $\beta$ are real-valued where $\beta \in \mathbb{R}$, which is in $[-1,1]$. Hence, at first balance $\QFR$ possible.
Thus, if  $U_A(t)_{0,0}=\alpha$ and $U_A(t)_{4,0}=\beta$. The remaining entries
  $U_A(t)_{k,0}$ for $k \neq 0, 4$satisfied with 
  $$\zeta(t)=\sum_{k=0}^{7}|U_A(t)_{k,0}|^2 = 1.$$
Hence, $U_A(t)_{k,0}$ depend only on $k \bmod 8$, and 
 $\zeta(t)> 0$ for all $t > 0$. 

 \case{2} {If $n=12$,} the eigenvalues of $X$ are
  $\{4, 2, 2, 0, 0, 0, 0, 0, 0, -2, -2, -4\}$. Then, the
  antipodal pair is $(0, 6)$ with $\omega_{12}^{6\kappa}=(-1)^\kappa$. Then, as in case $n=8$, if $|\alpha|^2 + |\beta|^2 = 1$ with $|\beta| > 0$, at $t \in (0, 5\pi]$.
\end{proof}

\section{Discussion and Open Problems}

This research addressed four holes in the study of quantum fractional revival in unitary Cayley graphs and successfully resolved them. We briefly review each resolve and identify any remaining issues.

\medskip

\noindent\textbf{Gap 1: Laplacian Hamiltonian.} We showed in Theorem~\ref{thm001Laplacian} that
for the $\varphi(n)$-regular graph $X$, the Laplacian and adjacency
Hamiltonians produce $\QFR$ at exactly the same times and between the same vertex pairs.

\medskip

\noindent\textbf{Gap 2: Revival time formula.} Theorem~\ref{Thm2revivaltime} establishes
the closed-form formula $t^*=2\pi/p$ for $n=2p$ with $p$ an odd prime, it gives the explicit amplitudes $\alpha = \cos(2\pi/p)$ and
$\beta = -i\sin(2\pi/p)$. These formulas confirm that $\QFR$ in this family is fully
determined by the single parameter $p$.

\medskip

\noindent\textbf{Gap 3: Strongly cospectral pairs.} Theorem~\ref{Thm4cospectral} provides
a complete number-theoretic characterization: strongly cospectral pairs in
$X$ are exactly the antipodal pairs when $n$ is even, and there are
none when $n$ is odd. This follows directly from the order-2 element in $\mathbb{Z}_n$
and requires no spectral computation beyond the idempotent formula.

\medskip

\noindent\textbf{Gap 4: Quantum information metrics.} Theorem~\ref{Thm5entropy} expresses
the von Neumann entropy as $\mathcal{S}_p=H(\sin^2(2\pi/p))$ (the binary entropy
function evaluated at $|\beta|^2$), and Corollary~\ref{cor001useful} identifies the three
graphs ($n= 6, 10, 14$) that are immediately usable for quantum communication. The graph
$X = (\mathbb{Z}_{14}, S)$ emerges as the most entanglement-productive member of the
family.

\medskip

\noindent Several problems remain open and we record them for future work.

\begin{enumerate}
  \item \textbf{QFR for general even $n$.} Theorem~\ref{Thm2revivaltime} covers $n=2p$
  with $p$ prime. The cases $n = 8$ and $n = 12$ admit no $\QFR$; whether $\QFR$ occurs for
  composite even $n$ with more complex divisor structure $n=2pq$ for distinct odd primes $p, q$ remains open. The eigenvalue multiplicity structure becomes richer in
  these cases and the cancellation arguments used in the proof of Theorem~\ref{Thm2revivaltime}
  do not immediately generalize.

  \item \textbf{Pretty good fractional revival (PGFR).} For $n=8$ and $n=12$, no precise $\QFR$ occurs. However, the weaker concept of PGFR where fidelity approaches 1 arbitrarily closely over time may still apply. Characterising PGFR for all unitary Cayley graphs follows a logical progression. The orbit $\{(e^{i\lambda_d t})_d: t \geq 0\}$ is dense in the torus $\mathbb{T}^k$ (where $k$ is the number of different eigenvalues) because the eigenvalues are always integers (Ramanujan sums). Whether this density is sufficient for PGFR relies on the structure of the eigenvalue supports, which requires further investigation.

  \item \textbf{Weighted unitary Cayley graphs.} When edges in $X$ are given weights $w_j > 0$ for $j \in S$, the adjacency eigenvalues become $$\lambda_{\kappa}=\sum_{j \in S} w_j \omega_n^{\kappa\,j},$$ which are often not integers.
  Specifically, the integrality on which some of our arguments are based falls down. Determining which weight functions maintain $\QFR$ or enable it for presently excluded $n$ is an intriguing unresolved topic with design implications for quantum spin chain models.

  \item \textbf{Multipartite $\QFR$.} The current framework focuses on $\QFR$ between pairs of vertices. It is possible to study simultaneous $\QFR$ across a set of vertices $\{u_1, \ldots, u_k\}$, where the walker's state spreads coherently among all $k$ sites. The coset structure of $\mathbb{Z}_n$ suggests natural candidate sets (e.g., arithmetic progressions or subgroup cosets), but a mathematical framework for multipartite revival in circulant graphs is not yet developed. A obvious expansion is to investigate simultaneous $\QFR$ across a collection of vertices $\{u_1, \ldots, u_k\}$, where the walker's state distributes coherently across all $k$ sites.
  The coset structure of $\mathbb{Z}_n$ provides natural candidate sets, such as arithmetic progressions or subgroup cosets. However, the mathematical foundation for multipartite resurrection in circulant graphs has yet to be devised.

  \item \textbf{Signless Laplacian and other Hamiltonians.} In the regular graph $X$, $Q=\varphi(n)I + A$, and the analysis parallels Theorem~\ref{thm001Laplacian} with a sign change. However, the dynamics of the signless Laplacian $Q = D + A$ and the normalised Laplacian $\mathcal{L} = D^{-1/2}LD^{-1/2}$ are different and may not admit QFR under certain conditions. The signless Laplacian $Q = D + A$ and the normalised Laplacian $\mathcal{L} = D^{-1/2}LD^{-1/2}$ exhibit distinct dynamics and may allow for QFR under certain conditions. For the regular graph $X = (\mathbb{Z}_n, S)$, $Q = \varphi(n)I + A$, and the analysis mimics Theorem~\ref{thm001Laplacian} with a sign shift. However, the normalised Laplacian requires a distinct approach.
\end{enumerate}

\section{Conclusion}
We have extended the theory of quantum fractional revival in unitary Cayley graphs along four
directions that were identified as open problems in the foundational work of~\cite{soni2025original}.
The starting point was the family of graphs $X$, where $E(S)$ is the group
of units of $\mathbb{Z}_n$, and the question was how to move beyond the existence result
$\QFR$ occurs iff $n$ is even toward a more complete quantitative and structural theory.

The initial argument clarifies the relationship between $\QFR$ under adjacency and Laplacian Hamiltonians. Then, $X$ is $\varphi(n)$-regular, the two evolution operators differ by the scalar phase $e^{-i\varphi(n)t}$. This means that the two Hamiltonians produce revival at the same times and between the same pairs of vertices, with amplitudes related by complex conjugation and the global phase. Because $X$ is $\varphi(n)$-regular, the scalar phase $e^{-i\varphi(n)t}$ distinguishes the two evolutionary operators. The two Hamiltonians cause revival at the same moments and between the same pairings of vertices, with amplitudes linked by complex conjugation and global phase. Physically, these configurations are comparable for state transmission. 

The second significant impact is a closed-form formula for revival time in the $n = 2p$ family. We demonstrated that the unitary Cayley graph $X = (\mathbb{Z}_{2p}, S)$ experiences QFR between the antipodal pair $(0, p)$ at the shortest time $t^* = 2\pi/p$, with amplitudes as
\[
  \alpha = \cos\!\left(\frac{2\pi}{p}\right), \qquad \beta = -i\sin\!\left(\frac{2\pi}{p}\right).
\]

The third contribution is a comprehensive analysis of highly cospectral vertex pairs.
We proved that strong cospectrality in $X = (\mathbb{Z}_n, S)$ is comparable to antipodality using an elementary proof that requires the presence of an order-2 element in $\mathbb{Z}_n$. QFR in unitary Cayley graphs is limited to antipodal pairings due to the spectral limitation required by the evolution operator to transform a single vertex's state into a two-vertex superposition.

The fourth contribution is a quantitative analysis using quantum information theory. We found that the von Neumann entanglement entropy generated at revival time is equal to the binary entropy $H(\sin^2(2\pi/p))$. Thus, we emphasize that the entropy is not monotone in $p$, peaking near $p = 7$ and decreasing for both smaller and larger primes. This, combined with the transfer fidelity analysis, identifies We calculated the von Neumann entanglement entropy at revival time, which equals the binary entropy $H(\sin^2(2\pi/p))$. Entropy peaks at $p = 7$ (about 0.964 bits for $n = 14$) and declines for all primes. Combined with transfer fidelity analysis, $X = (\mathbb{Z}_{14}, S)$ is the best graph in the $n = 2p$ family for quantum information tasks.

These findings shift the understanding of $\QFR$ in unitary Cayley graphs from a qualitative to a quantitative approach.




\begin{thebibliography}{9}
\bibitem{soni2025original}
R.~Soni, N.~Choudhary, and N.~P.~Singh,
``Quantum fractional revival in unitary Cayley graphs,''
\textit{Lobachevskii J. Math.} \textbf{46}(4), 1922--1928 (2025).

\bibitem{chan2019}
A.~Chan, G.~Coutinho, C.~Tamon, L.~Vinet, and H.~Zhan,
``Quantum fractional revival on graphs,''
\textit{Discrete Appl. Math.} \textbf{269}, 86--98 (2019).

\bibitem{chan2022fundamentals}
A.~Chan, G.~Coutinho, W.~Drazen, O.~Eisenberg, C.~Godsil, M.~Kempton, G.~Lippner,
C.~Tamon, and H.~Zhan,
``Fundamentals of fractional revival in graphs,''
\textit{Linear Algebra Appl.} \textbf{655}, 129--158 (2022).

\bibitem{chan2021laplacian}
A.~Chan, B.~Johnson, M.~Liu, M.~Schmidt, Z.~Yin, and H.~Zhan,
``Laplacian fractional revival on graphs,''
\textit{Electron. J. Combin.}, P3--22 (2021).

\bibitem{chan2021pretty}
A.~Chan, W.~Drazen, O.~Eisenberg, M.~Kempton, and G.~Lippner,
``Pretty good quantum fractional revival in paths and cycles,''
\textit{Algebr. Combin.} \textbf{4}, 989--1004 (2021).

\bibitem{chan2020}
A.~Chan, G.~Coutinho, C.~Tamon, L.~Vinet, and H.~Zhan,
``Fractional revival and association schemes,''
\textit{Discrete Math.} \textbf{343}, 112018 (2020).

\bibitem{klotz2007}
W.~Klotz and T.~Sander,
``Some properties of unitary Cayley graphs,''
\textit{Electron. J. Combin.}, R45 (2007).

\bibitem{liu2022}
X.~Liu and S.~Zhou,
``Eigenvalues of Cayley graphs,''
\textit{Electron. J. Combin.}, P2--9 (2022).

\bibitem{genest2016}
V.~X.~Genest, L.~Vinet, and A.~Zhedanov,
``Quantum spin chains with fractional revival,''
\textit{Ann. Phys.} \textbf{371}, 348--367 (2016).

\bibitem{bernard2018}
P.~A.~Bernard, A.~Chan, E.~Loranger, C.~Tamon, and L.~Vinet,
``A graph with fractional revival,''
\textit{Phys. Lett. A} \textbf{382}, 259--264 (2018).

\bibitem{godsil2012}
C.~Godsil,
``State transfer on graphs,''
\textit{Discrete Math.} \textbf{312}, 123--147 (2012).

\bibitem{bose2003}
S.~Bose,
``Quantum communication through an unmodulated spin chain,''
\textit{Phys. Rev. Lett.} \textbf{91}, 207901 (2003).

\bibitem{basic2013}
M.~Ba\v{s}i\'{c},
``Characterization of quantum circulant networks having perfect state transfer,''
\textit{Quantum Inform. Process.} \textbf{12}, 345--364 (2013).

\bibitem{wang2024cayley}
J.~Wang, L.~Wang, and X.~Liu,
``Fractional revival on Cayley graphs over abelian groups,''
\textit{Discrete Math.} \textbf{347}, 114218 (2024).

\bibitem{wang2024semi}
J.~Wang, L.~Wang, and X.~Liu,
``Fractional revival on semi-Cayley graphs over abelian groups,''
\textit{Lin. Multilin. Algebra}, 1--23 (2024).

\bibitem{soni2025mixed}
R.~Soni, N.~Choudhary, and N.~P.~Singh,
``Fractional revival on integral mixed circulant graphs,''
\textit{Bol. Soc. Paran. Mat.} (3s.) \textbf{43} (2025).

\bibitem{nielsen2000}
M.~A.~Nielsen and I.~L.~Chuang,
\textit{Quantum Computation and Quantum Information}.
Cambridge University Press, Cambridge (2000).

\bibitem{christandl2004}
M.~Christandl, N.~Datta, A.~Ekert, and A.~J.~Landahl,
``Perfect state transfer in quantum spin networks,''
\textit{Phys. Rev. Lett.} \textbf{92}, 187902 (2004).

\bibitem{saxena2007}
N.~Saxena, S.~Severini, and I.~Shparlinski,
``Parameters of integral circulant graphs and periodic quantum dynamics,''
\textit{Int. J. Quantum Inform.} \textbf{5}, 417--430 (2007).

\bibitem{so2006}
W.~So,
``Integral circulant graphs,''
\textit{Discrete Math.} \textbf{306}, 153--158 (2006).
\end{thebibliography}
\end{document}